\documentclass[12pt]{article}

\usepackage[T1]{fontenc}
\usepackage[margin=1.35in]{geometry}
\usepackage{amsmath,amssymb,amsthm,booktabs,microtype}
\usepackage{cite}
\usepackage{hyperref}
\usepackage{xcolor}\hypersetup{colorlinks,citecolor=cyan!50!blue!50!white,linkcolor=cyan!50!blue!50!white,urlcolor=cyan!50!blue!50!white}
\newcommand{\arXiv}[1]{arXiv:\href{https://arxiv.org/abs/#1}{#1}}
\usepackage{doi}
\usepackage{titlesec}
\titleformat{\section}
  {\centering\normalfont\normalsize}
  {\thesection.} 
  {0.1in}
  {\bfseries}
\titleformat{\subsection}
  {\normalfont\normalsize}
  {\thesubsection.} 
  {0.1in}
  {\bfseries}

\newcommand{\R}{\mathbb R}
\newcommand{\tr}{\operatorname{tr}}
\newcommand{\vol}{\operatorname{vol}}
\newcommand{\area}{\operatorname{area}}
\newcommand{\Sym}{\operatorname{Sym}}
\newcommand{\Arch}{P_{\mathrm{Arch}}}
\newcommand{\Ical}{\mathcal I}
\newcommand{\mb}[1]{\mathbf{#1}}
\newcommand{\e}{\mathrm e}
\newcommand{\K}{\mathsf K}
\newcommand{\F}{\mathsf F}
\newcommand{\br}[1]{\langle{#1}\rangle}

\newtheorem{theorem}{Theorem}
\newtheorem{proposition}[theorem]{Proposition}
\newtheorem{lemma}[theorem]{Lemma}

\title{\Large Local stability for the truncated octahedral conjecture}

\author{\normalsize Lark Song}
\date{}

\begin{document}
\maketitle

\begin{abstract}
This note proves that the truncated octahedron of Archimedes is a strict local minimizer of the isoperimetric quotient among parallelohedra in three dimensions. This extends the corresponding result from the five-dimensional lattice-Voronoi subclass to the full ten-dimensional local parameter space. A six-zone parametrization and Schur's lemma reduce the Hessian to two positive definite $2\times 2$ matrices.
\end{abstract}

\section{Introduction}

A \emph{parallelohedron} is a convex polyhedron that admits a face-to-face tiling of Euclidean three-space by translations.  
Minkowski proved that every parallelohedron, as well as each of its facets, is centrally symmetric \cite{Minkowski1897}.
Venkov proved, and McMullen later proved independently, that the face-to-face hypothesis can be dropped \cite{Venkov1954,McMullen1980}. 
Fedorov proved that there are five combinatorially distinct types of parallelohedra: parallelepipeds, hexagonal prisms, rhombic dodecahedra, elongated dodecahedra, and truncated octahedra \cite{Fedorov1885}. 

For a convex body $P\subset\R^3$, we write its isoperimetric quotient
\[
  \Ical(P):=\frac{\area(\partial P)}{\vol(P)^{2/3}}.
\]
The truncated octahedral conjecture asks whether the Archimedean truncated octahedron minimizes this quotient among all parallelohedra in $\R^3$ \cite{Bezdek2006}.  
L\'angi proved the analogous result for mean width \cite{Langi2022}.  
Using Selling's parametrization of ternary quadratic forms \cite{Selling1874}, Cesaroni--Novaga recently proved strict local minimality at fixed volume within the subclass of lattice Voronoi cells \cite{CesaroniNovaga2026}. Voronoi conjectured that every parallelohedron is affinely equivalent to a lattice Voronoi cell \cite{Voronoi1908}, which is now known through dimension five \cite{Garber2025}. In dimension three the affine description is classical \cite{Langi2022}, but it does not settle Bezdek's conjecture locally, since $\Ical$ is invariant under similarities but not under general affine maps.

Modulo similarities, the family of lattice Voronoi cells is five-dimensional, whereas the full local parameter space at $\Arch$ is ten-dimensional.  In the chart below, five parameters rescale the six edge zones modulo common dilation, and five describe an independent determinant-one affine metric modulo rotations.  
While the Cesaroni--Novaga result settles stability within the five-dimensional Voronoi subclass, analyzing the full parameter space requires controlling the five transverse directions along with the mixed second variations.  We prove that the full Hessian is positive definite and obtain a quadratic estimate in volume-normalized Hausdorff distance modulo rotations.

Let $\mathcal P_3$ be the family of parallelohedra in $\R^3$ whose centers of symmetry are at the origin, and fix an Archimedean truncated octahedron $\Arch\in\mathcal P_3$.  For every centered convex body $K\subset\R^3$, define the volume-normalized representative of $K$ relative to $\Arch$
\[
 \widehat K:=\left(\frac{\vol(\Arch)}{\vol(K)}\right)^{1/3}K.
\]
Define
\[
 d_{\mathrm{sim}}(P,\Arch):= \inf_{Q\in\mathrm{SO}(3)}d_H(Q\widehat P,\Arch).
\]
We call two bodies \emph{similar} if they differ by a rotation and a positive dilation.
\begin{theorem}\label{thm:main}
There exist a Hausdorff neighborhood $\mathcal U$ of $\Arch$ in $\mathcal P_3$ and a constant $c_{\mathrm{sim}}>0$ such that
\begin{equation}\label{eq:intrinsic-stability}
 \Ical(P)-\Ical(\Arch)\geq
 c_{\mathrm{sim}}d_{\mathrm{sim}}(P,\Arch)^2.
\end{equation}
For $P\in\mathcal U$, equality holds if and only if $P$ is similar to $\Arch$. 
\end{theorem}
On the labeled similarity slice introduced below, the proof also shows that the pullback of $\Ical$ has positive definite Hessian at the origin.

\subsection{Proof outline}

The other four Fedorov types have at most twelve facets, so none can converge to the fourteen-facet Archimedean cell.  Consequently, it suffices to work locally in the truncated-octahedral class.  The six-zone normal form identifies its labeled similarity space near $\Arch$ with
\[
 \mathcal W_0\oplus\mathcal G_0,
\]
with $\dim\mathcal W_0=\dim\mathcal G_0=5$,
where $\mathcal W_0$ records zone rescalings and $\mathcal G_0$ records affine metrics.
We verify that this chart recovers the Hausdorff topology, which allows the coordinate Hessian estimate to be transferred to Hausdorff-near convex bodies.  Tetrahedral $S_4$-symmetry then gives
\[
 \mathcal W_0\oplus\mathcal G_0\simeq2E\oplus2T,
\]
where $E=[2,2]$ and $T=[3,1]$. 

Schur's lemma reduces the Hessian to one $2\times2$ multiplicity matrix on the two copies of $E$ and one on the two copies of $T$.  Two explicit two-parameter families determine both matrices, and their determinants are positive.  Taylor's theorem then gives a quadratic coordinate gap, which the support function transfers to the asserted Hausdorff estimate.

We write $\Sym_3(\R)$ for symmetric $3\times3$ matrices and $\|\cdot\|_F$ for the Frobenius norm.  
Boldface is reserved for vectors attached to the fixed tetrahedral frame.
All representations are over $\R$.

\section{The six-zone similarity space}

\subsection{Normal form and uniqueness}

We call a parallelohedron \emph{truncated-octahedral} if it is combinatorially equivalent to $\Arch$.  An \emph{edge zone} is a class of mutually parallel edges. 
For the zonotopes below, the six zones correspond to the six generator directions.  
The six-zone normal form gives a Hausdorff-compatible chart on the local similarity quotient of this class.
Fix the normalized frame
\[
\begin{aligned}
 \mb n_1&=2^{-2/3}(1,1,1), \\
 \mb n_2&=2^{-2/3}(1,-1,-1),\\
 \mb n_3&=2^{-2/3}(-1,1,-1),\\
 \mb n_4&=2^{-2/3}(-1,-1,1).
\end{aligned}
\]
Then $\mb n_1+\cdots+\mb n_4=0$ and
$\det(\mb n_1,\mb n_2,\mb n_3)=1$.  Write
$\mb n_{ij}^{\times}=\mb n_i\times\mb n_j$ and $\beta_{ji}=\beta_{ij}$.
The fixed-frame form of the
six-zone parametrization is
\begin{equation}\label{eq:normal-form}
 Z(\beta,M):=\sum_{1\leq i<j\leq4}
 \left[-\frac12\beta_{ij}M\mb n_{ij}^{\times},\,
             \frac12\beta_{ij}M\mb n_{ij}^{\times}\right]
\end{equation}
with $\beta_{ij}>0$ and $M\in\mathrm{SL}(3,\R)$.
These are precisely the centered truncated-octahedral parallelohedra.  Indeed, L\'angi's
parametrization \cite{Langi2022} allows an arbitrary centered tetrahedral frame
$\mb v_i=A\mb n_i$ with
$A\in\mathrm{GL}(3,\R)$.  
Let
\[
 M_0:=(\operatorname{adj}A)^T=\det(A)A^{-T}.
\]
Then
\[
(A\mb n_i)\times(A\mb n_j)=M_0\mb n_{ij}^{\times} \quad\text{and}\quad
\det M_0=(\det A)^2>0.
\]
Dividing $M_0$ by $(\det M_0)^{1/3}$ and absorbing this positive scalar into the six weights gives this normal form. 
Replacing L\'angi's $[0,g]$ segments by centered segments merely translates the zonotope.  
Conversely,
$Z(\mb{1},I)$ tiles, invertible linear maps preserve translative tilings, and McMullen's
equivalence lemma permits independent positive rescaling of the six generating segments
\cite{McMullen1975}.

After a fixed rotation and positive dilation, we henceforth choose
\[
 \Arch:=Z(\mb{1},I).
\]
This entails no loss of generality: changing the fixed representative only rescales
$d_{\mathrm{sim}}$ and the stability constant.
The fixed zone directions satisfy
\begin{equation}\label{eq:line-relations}
 \mb n_{14}^{\times}=-\mb n_{12}^{\times}-\mb n_{13}^{\times},\quad
 \mb n_{24}^{\times}=\mb n_{12}^{\times}-\mb n_{23}^{\times},\quad
 \mb n_{34}^{\times}=\mb n_{13}^{\times}+\mb n_{23}^{\times}.
\end{equation}

We call a relabeling \emph{admissible} if the relabeled generators again occur in a
normal-form presentation.

\begin{lemma}\label{lem:zone-uniqueness}
The six generating segments of a truncated-octahedral zonotope are determined by
its six labeled edge zones.  Two normal-form presentations of the same
labeled body coincide, and every admissible relabeling of the six zones is induced by a
unique element of $S_4$.
\end{lemma}

\begin{proof}
An exposed face of a Minkowski sum is the sum of the corresponding exposed faces of its
summands.  Since the six generator directions are pairwise nonparallel, every edge is
therefore a translate of exactly one generating segment.  Thus a labeled edge zone
determines both the direction and the length of its generating segment.  Conversely, for
each generator one can choose an exposing vector perpendicular to it and to no other
generator, so every generating segment occurs as an edge.

Suppose first that the labels are fixed, and write $$A:=M^{-1}M'$$ for the relative linear map
between two presentations.  Equality of the labeled generator segments implies that $A$
preserves every line $\R\mb n_{ij}^{\times}$.  It is diagonal in the basis
$\mb n_{12}^{\times},\mb n_{13}^{\times},\mb n_{23}^{\times}$, and the three displayed relations
force its diagonal entries to agree.  Since both linear
factors have determinant one, this common scalar is one, and equality of the generating
segments then gives equality of all six weights.

It remains to identify the possible relabelings.  Use
$\mb n_{12}^{\times},\mb n_{13}^{\times},\mb n_{23}^{\times}$ as an ordered basis whose basis matrix has determinant one.  In this basis the other three vectors are
\[
 \mb n_{14}^{\times}=-e_1-e_2,\quad
 \mb n_{24}^{\times}=e_1-e_3,\quad
 \mb n_{34}^{\times}=e_2+e_3.
\]
The corresponding $3\times3$ determinants vanish precisely for the four stars
\[
 D_i=\{\R\mb n_{ij}^{\times}:j\ne i\}\quad\text{for } i=1,\ldots,4.
\]
Every linear automorphism of the line configuration permutes these four triples.  Moreover,
$\R\mb n_{ij}^{\times}$ is the unique line in $D_i\cap D_j$, so the induced permutation of the
$D_i$ determines the relabeling and belongs to $S_4$.  Conversely, let
$Q_\pi\in\mathrm O(3)$ satisfy $Q_\pi\mb n_i=\mb n_{\pi(i)}$.  
Then
\[
L_\pi:=(\det Q_\pi)Q_\pi\in\mathrm{SO}(3) \quad\text{and}\quad L_\pi(\mb n_i\times\mb n_j)=\mb n_{\pi(i)}\times\mb n_{\pi(j)}.
\]
Hence every element of $S_4$ is admissible.
\end{proof}

\subsection{Area, volume, and logarithmic coordinates}

Write $G:=M^{-1}M^{-T}$, and let $\K_4$ be the complete graph on
$\{1,2,3,4\}$.  The zonotope volume formula \cite{Shephard1974} and the
surface-area identity for a zonotope with generator vectors $g_e$
\cite{Langi2022} give
\begin{equation}\label{eq:zonotope-formulas}
\begin{aligned}
 V(\beta)&:=\vol Z(\beta,M)
 =\sum_{J\text{ spanning tree of }\K_4} \; \prod_{e\notin J}\beta_e,\\
 S(\beta,G)&:=\area(\partial Z(\beta,M))
 =2\sum_{b\in\F}\gamma_b(\beta)\sqrt{\mb n_b^TG\mb n_b}.
\end{aligned}
\end{equation}
The nonzero generator triples are precisely the complements of spanning trees, and the
corresponding triples of unweighted frame vectors $\mb n_{ij}^{\times}$ have determinants
of absolute value one.  Index the seven facet-normal lines by
\[
 \F=\{\br{1},\br{2},\br{3},\br{4}\}
 \sqcup\{\br{12|34},\br{13|24},\br{14|23}\}.
\]
Here $\br{i}$ labels the partition $\{i\}\mid(\{1,2,3,4\}\setminus\{i\})$, and
$\br{ij|pq}$ labels $\{i,j\}\mid\{p,q\}$.  
The first four labels correspond to the hexagonal facet pairs and the
last three to the parallelogram facet pairs.
Grouping the pairwise
cross products in the zonotope surface-area formula according to these seven normal lines
uses the identities
\[
 (\mb n_i\times\mb n_p)\times(\mb n_i\times\mb n_q)=\pm\mb n_i
 \quad\text{and}\quad
 (\mb n_i\times\mb n_j)\times(\mb n_p\times\mb n_q)
   =\pm(\mb n_i+\mb n_j),
\]
where $i,p,q$ are distinct in the first identity and $\{i,j,p,q\}=\{1,2,3,4\}$ in the
second.  Thus the facet data are
\[
\begin{alignedat}{3}
 \mb n_{\br{i}} &= \mb n_i,
 &\qquad \gamma_{\br{i}}
   &= \sum_{1\le p<q\le4,\; p,q\ne i}
        \beta_{ip}\beta_{iq},
 &\qquad &\text{for } i=1,\ldots,4,\\
 \mb n_{\br{ij|pq}} &= \mb n_i+\mb n_j,
 &\qquad \gamma_{\br{ij|pq}}
   &= \beta_{ij}\beta_{pq},
 &\qquad &\text{for } \br{ij|pq}\in\F.
\end{alignedat}
\]
Since $\det M=1$, volume is independent of $M$.  At $\beta=\mb{1}$ and $G=I$,
\[
\begin{aligned}
 V_0:={}&V(\mb{1})=\vol(\Arch)=16,\\
 S_0:={}&S(\mb{1},I)=6\cdot 2^{1/3}(1+2\sqrt3),\\
 \Ical(\Arch)={}&\frac{3(1+2\sqrt3)}{4^{2/3}}.
\end{aligned}
\]
Remove dilation and rotation by the exponential slice
\[
\begin{alignedat}{3}
 \beta_{ij}&=\e^{\ell_{ij}} &\quad\text{with}\quad
 \ell &\in\mathcal W_0=
 \{\ell\in\R^6:\sum_{i<j}\ell_{ij}=0\},\\
 G&=\e^\sigma &\quad\text{with}\quad
 \sigma &\in\mathcal G_0=
 \{\sigma\in\Sym_3(\R):\tr\sigma=0\}.
\end{alignedat}
\]
Here exponentiation of $\ell$ is coordinatewise.  The condition $\sum_{i<j}\ell_{ij}=0$
fixes $\prod_{i<j}\beta_{ij}=1$, which removes dilation while allowing the physical
volume to vary.  We impose volume normalization separately in the Hausdorff estimate.
Take $M=\e^{-\sigma/2}$ and define
\[
 \Phi(\ell,\sigma):=Z(\e^\ell,\e^{-\sigma/2})
 \quad\text{and}\quad
 \Psi(\ell,\sigma):=\frac{S(\e^\ell,\e^\sigma)}{V(\e^\ell)^{2/3}}.
\]
Indeed, $G^{-1}=M^TM$, so polar decomposition gives $M=QG^{-1/2}$ with $Q\in\mathrm{SO}(3)$. 
Modulo left rotations, the choice $M=G^{-1/2}$ is unique.
For the affine factor, the same polar exponential construction applies in every dimension, with a symmetric trace-free logarithm.

With the facet structure fixed, the individual coordinate volume and surface-area functionals are real analytic near
the origin.  Indeed, the volume formula expresses $V(\e^\ell)$ as a finite
sum of exponential monomials.  Each facet weight
$\gamma_b(\e^\ell)$ is also real analytic.  The matrix exponential is real analytic,
and each radicand $\mb n_b^T\e^\sigma\mb n_b$ equals $|\mb n_b|^2>0$ at
$\sigma=0$.  All seven radicands therefore remain positive near the origin, where
their positive square roots are real analytic.  The surface-area formula then shows that $S(\e^\ell,\e^\sigma)$ is real analytic
as well.  Since $V(\mb{1})=16>0$, the coordinate volume stays positive near the
origin, so the quotient $\Psi=\Ical\circ\Phi$ is real analytic, in particular $C^2$,
on a neighborhood of the origin in $\mathcal W_0\oplus\mathcal G_0$.
It represents $\Ical$ on the labeled similarity space.
Write $h=(\ell,\sigma)$, $\Phi(h):=\Phi(\ell,\sigma)$, and
$|h|^2:=|\ell|^2+\|\sigma\|_F^2$.
The natural $S_4$ action is
\[
 (\pi\ell)_{\pi(i)\pi(j)}=\ell_{ij}
 \quad\text{and}\quad \pi\sigma=Q_\pi\sigma Q_\pi^T,
\]
where the edge coordinates are indexed by unordered pairs and
$Q_\pi\in\mathrm O(3)$ is determined by $Q_\pi\mb n_i=\mb n_{\pi(i)}$.  If $\pi$ is odd,
$Q_\pi$ is improper, but replacing it by $-Q_\pi$ gives a rotation and leaves the
conjugation action unchanged.  Thus $\Phi(\pi h)$ is congruent to $\Phi(h)$, and the action
preserves $\Psi$ and the norm $|h|$.

\subsection{Compatibility with Hausdorff convergence}

\begin{lemma}\label{lem:open-class}
There is a Hausdorff neighborhood of $\Arch$ in $\mathcal P_3$ that consists entirely of
truncated-octahedral parallelohedra.
\end{lemma}

\begin{proof}
By Fedorov's classification \cite{Fedorov1885}, every three-dimensional
parallelohedron belongs to one of five types.  Each of the other four types has at most
twelve facets, whereas $\Arch$ has fourteen.  If a sequence of parallelohedra of the other four types converged to $\Arch$,
its surface-area measures would converge weakly to that of $\Arch$
\cite{Schneider2014}.  After passing to a subsequence, each measure with at
most twelve atoms has a weak limit with at most twelve atoms: pad the atomic lists with
zero masses, then use compactness of $S^2$ and boundedness of the total masses.  The
surface-area measure of $\Arch$ has fourteen positive atoms, a contradiction.
\end{proof}

\begin{lemma}\label{lem:simple-stability}
Let $K_k\to K$ in Hausdorff distance be three-dimensional polytopes whose facets are
bijectively labeled by the same finite set $A$, and assume that every outer unit normal and support
number converges to its counterpart for $K$.  If $K$ is simple, then, for all large $k$,
the vertex-facet incidences of $K_k$ and $K$ agree.  The correspondingly labeled vertices
and edges then converge.
\end{lemma}

\begin{proof}
Write the facet inequalities of $K$ as
\[
 \langle n_a,x\rangle\le s_a \quad\text{for } a\in A,
\]
and write $n_a^{(k)},s_a^{(k)}$ for the corresponding data of $K_k$.  If $v$ is a vertex
of $K$, its three incident normals have nonzero determinant.  Their perturbed hyperplanes
therefore meet in a point $v_k\to v$.  Every nonincident inequality has strictly positive
slack at $v$. 
Since there are only finitely many vertex-facet pairs, all these strict inequalities persist uniformly.  Hence every vertex of $K$ gives a vertex of $K_k$ with the same incident facets.

Conversely, suppose that $K_k$ had an additional vertex along a subsequence.  Choose three incident facets with linearly independent normals.  
After passing to a further subsequence, their labels
are fixed.  
Hausdorff convergence bounds the vertices, so they converge to a point of $K$ lying on those three
facets.  Simplicity makes that point a vertex of $K$ and makes those three labels precisely
its incident facets.  The corresponding perturbed hyperplanes have a unique intersection,
namely the vertex already constructed, a contradiction.  Thus the vertex-facet incidences
agree.  Vertices are continuous
intersections of their three incident hyperplanes, and the same is consequently true for
the labeled edges.
\end{proof}

The similarity quotient carries the metric
\[
 \delta_{\mathrm{sim}}([P],[R]):=
 \inf_{Q\in\mathrm{SO}(3)}d_H(Q\widehat P,\widehat R).
\]
Thus $d_{\mathrm{sim}}(P,\Arch)=\delta_{\mathrm{sim}}([P],[\Arch])$.
Compactness of $\mathrm{SO}(3)$ shows that this distance vanishes exactly for similar
bodies.  Central symmetry makes the $\mathrm{SO}(3)$- and $\mathrm O(3)$-quotients
identical, since $QK=(-Q)K$ and $-Q\in\mathrm{SO}(3)$ when $Q$ is improper.

\begin{proposition}\label{prop:topology}
The map $h\mapsto[\Phi(h)]$ induces a local homeomorphism from a neighborhood of $[0]$ in
$(\mathcal W_0\oplus\mathcal G_0)/S_4$ onto a neighborhood of $[\Arch]$ in the
similarity quotient of the truncated-octahedral class.  In particular, $[P_k]\to[\Arch]$
if and only if, after relabeling the six zones, the product-normalized representatives are
$\Phi(h_k)$ with $h_k\to0$.
\end{proposition}

\begin{proof}
Forward continuity follows from the support function of \eqref{eq:normal-form}.  For the
converse, choose volume-normalized representatives and rotations for which
$P_k\to\Arch$ in Hausdorff distance.  Their support functions converge uniformly, and
their surface-area measures converge weakly
\cite{Schneider2014}.

We first recover compatible labeled generators.
Choose disjoint spherical caps around the fourteen facet normals of $\Arch$, with boundaries carrying no limiting surface-area mass.  Weak convergence makes the mass of each cap positive for all large $k$, so every cap contains at least one facet normal of $P_k$.
The caps are disjoint and $P_k$ has exactly fourteen facets, so every cap contains exactly one facet normal and none lie outside their union.
Because the caps may be chosen with arbitrarily small radii, the uniquely matched facet normals converge to those of $\Arch$. 
Uniform support-function convergence gives convergence of their support numbers.  Since $\Arch$ is simple, Lemma~\ref{lem:simple-stability} implies that
	the labeled vertices and edges converge as well.
Fix one edge of $\Arch$ in each of its six zones and take the corresponding edges of
$P_k$.  
Their limiting directions are pairwise distinct, so for large $k$ they belong to six distinct zones. 
Since a truncated-octahedral zonotope has exactly six zones, those six zones exhaust all its zones.
By Lemma~\ref{lem:zone-uniqueness}, each is a translate of its generating segment.
Denote the resulting oriented generators by $g_{ij}$, suppressing their dependence on
$k$ and choosing their signs by positive inner product with the limiting generators.  Then
$g_{ij}\to\mb n_{ij}^{\times}$.
Every nonstar triple of limiting zone lines has nonzero determinant, so the corresponding
triple of the $g_{ij}$ remains independent for all large $k$.  In any normal-form indexing
there are exactly four dependent triples.  Since all sixteen nonstar triples in the limiting
indexing remain independent, its four dependent triples have the same index sets as
$D_1,\ldots,D_4$.  Thus the change from a normal-form indexing to the limiting indexing
preserves the four-star incidence pattern.  The argument of
Lemma~\ref{lem:zone-uniqueness} shows that it is induced by a unique element of $S_4$.
Hence the limiting labels are admissible, up to the $S_4$ action.

We next reconstruct the weights and affine metric from these generators.
Since the generating segments are unoriented, write
\[
g_{ij}=s_{ij}\beta_{ij}M\mb n_{ij}^{\times}
\]
with $s_{ij}\in\{\pm1\}$.
Substitution into \eqref{eq:line-relations}, together with
$g_{ij}\to\mb n_{ij}^{\times}$ and $\beta_{ij}>0$, shows successively that
\[
s_{14}=s_{12}=s_{13},\quad
s_{24}=s_{12}=s_{23},\quad
s_{34}=s_{13}=s_{23}.
\]
Thus all six signs are equal.  Since
$\det(g_{12},g_{13},g_{23})\to1$ and $\det M=1$, their common value is
$+1$.  Hence the chosen orientations agree with the normal-form
orientations.
With this labeling, the line relations become
\[
g_{14}=-\frac{\beta_{14}}{\beta_{12}}g_{12}
        -\frac{\beta_{14}}{\beta_{13}}g_{13},\quad
g_{24}= \frac{\beta_{24}}{\beta_{12}}g_{12}
        -\frac{\beta_{24}}{\beta_{23}}g_{23},\quad
g_{34}= \frac{\beta_{34}}{\beta_{13}}g_{13}
        +\frac{\beta_{34}}{\beta_{23}}g_{23}.
\]
The three coefficient pairs converge to $(-1,-1)$, $(1,-1)$, and $(1,1)$.  Hence,
for all sufficiently large $k$, their signs are as displayed and their absolute values
determine the positive weight ratios continuously.  Since these ratios connect all six
weights, the determinant identity
\[
 \beta_{12}\beta_{13}\beta_{23}
 =\frac{\det(g_{12},g_{13},g_{23})}
        {\det(\mb n_{12}^{\times},\mb n_{13}^{\times},\mb n_{23}^{\times})}>0
\]
fixes the remaining common scale and determines all six positive weights.  
The first three generators
then recover
\[
 M=\bigl[g_{12}/\beta_{12}\;g_{13}/\beta_{13}\;g_{23}/\beta_{23}\bigr]
   \bigl[\mb n_{12}^{\times}\;\mb n_{13}^{\times}\;\mb n_{23}^{\times}\bigr]^{-1},
\]
so $M$ also depends continuously on the generators and has determinant one.  Define
\[
 a:=\left(\prod_{i<j}\beta_{ij}\right)^{1/6}.
\]
Replacing $P_k$ by $a^{-1}P_k$ and every weight by $a^{-1}\beta_{ij}$
does not change its similarity class, leaves $M$ and $G=M^{-1}M^{-T}$ unchanged, and
normalizes the product of the six weights to one.  Thus, with this compatible
labeling, $a^{-1}\beta$ and $G$ vary continuously with the generators and converge to
$\mb{1}$ and $I$, respectively.  Consequently, the logarithmic coordinates
\[
 \ell_{ij}=\log(a^{-1}\beta_{ij})\quad\text{and}\quad \sigma=\log G
\]
converge to $(0,0)$.

We finally pass to the similarity quotient.  Polar decomposition gives $M=QG^{-1/2}$, so
\[
Q^{-1}(a^{-1}P_k)=\Phi(\ell,\sigma).
\]
Left rotations of $M$ leave $G$ unchanged,
while any two compatible zone labelings differ by the $S_4$ action of Lemma~\ref{lem:zone-uniqueness}.  For injectivity, suppose that $\Phi(h)$ and $\Phi(h')$ are similar.
Transport the zones by the similarity and then apply an $S_4$-relabeling.  Labeled uniqueness shows that their weights differ by the common dilation factor.  
Since both products of weights equal one, that factor is one.  Uniqueness in the polar decomposition gives $h'=\pi h$ for $\pi\in S_4$.

Choose a sufficiently small closed coordinate ball $\overline B$ centered at $0$, and let $B$ be its
interior.  The induced map from
$\overline B/S_4$ into the similarity quotient is continuous and injective.  Since its
source is compact and its target is metric, it is a homeomorphism onto its image.  The
preceding reconstruction shows that every sequence of truncated-octahedral classes
converging to $[\Arch]$ has, after relabeling, coordinates converging to zero.  Consequently, the image of
$B/S_4$ contains a neighborhood of $[\Arch]$.  Restricting to this neighborhood completes
the local-homeomorphism claim.
\end{proof}

\section{Hessian and stability}

Recall that the local coordinate is $h=(\ell,\sigma)$.
Here $\ell=\log\beta$ records relative zone rescalings, while $\sigma=\log G$ records the
affine metric $G=M^{-1}M^{-T}$, with $M=\e^{-\sigma/2}$ in our chosen representative.
The Archimedean cell corresponds to $h=0$.
Below, $(X,\Xi)$ are coordinates for the $E$- and $T$-components of $\ell$, while
$(U,\Sigma)$ are coordinates for the corresponding components of $\sigma$.
These are linear coordinates on the same ten-dimensional local parameter space.

The full ten-variable calculation is governed by tetrahedral symmetry.  We first identify
the two isotypic components and then compute one multiplicity matrix on each.

Throughout this section, derivatives of $S$ and $V$ are taken after the substitutions
$\beta=\e^\ell$ and $G=\e^\sigma$, within $\mathcal W_0\oplus\mathcal G_0$.
The $S_4$ action defined above preserves $\Psi$.  Neither $\mathcal W_0$ nor $\mathcal G_0$ contains a nonzero invariant vector.
Before imposing the zero-sum and trace-zero conditions, the only invariant weight direction is common scaling, and the only invariant symmetric matrix direction is the scalar direction.  Since $\Psi$ is invariant,
\[
 D\Psi(0)=0.
\]
The same argument applies separately to $S$ and $V$, so their first variations vanish.

Let $E=[2,2]$ and $T=[3,1]$ denote the two- and three-dimensional irreducible
representations.  In the conjugacy-class order $1,(12),(12)(34),(123),(1234)$, both 
edge representation and $\Sym^2T$ have character $(6,2,2,0,0)$.  Removing the trivial
summand yields
\[
 \mathcal W_0\simeq E\oplus T,
 \quad
 \mathcal G_0\simeq E\oplus T,
 \quad
 \mathcal W_0\oplus\mathcal G_0\simeq2E\oplus2T.
\]
Geometrically, the $E$-modes change the logarithmic weights equally within each
complementary zone pair $(12,34)$, $(13,24)$, $(14,23)$ and couple these changes to
diagonal trace-free metric variations.  The $T$-modes change the logarithmic weights
with opposite signs within each pair and couple them to off-diagonal metric variations.

In the edge order $(12,13,14,23,24,34)$, use the following coordinates:
\[
\begin{gathered}
\begin{alignedat}{2}
 \ell_E &= (x,y,z,z,y,x),
 &\quad& \text{where } X=(x,y,z)
          \;\text{ with }\; x+y+z=0,\\
 \sigma_E &= \operatorname{diag}(u,v,w),
 &\quad& \text{where } U=(u,v,w)
          \;\text{ with }\; u+v+w=0, \\
 \ell_T &= (\xi_1,\xi_2,\xi_3,-\xi_3,-\xi_2,-\xi_1),
 &\quad& \text{where } \Xi=(\xi_1,\xi_2,\xi_3),\\
 \sigma_T &=
 \begin{pmatrix}
  0       & \zeta_3 & \zeta_2\\
  \zeta_3 & 0       & \zeta_1\\
  \zeta_2 & \zeta_1 & 0
 \end{pmatrix},
 &\quad& \text{where } \Sigma=(\zeta_1,\zeta_2,\zeta_3).
\end{alignedat}
\end{gathered}
\]
Thus $\ell=\ell_E(X)+\ell_T(\Xi)$ and
$\sigma=\sigma_E(U)+\sigma_T(\Sigma)$.  The two isotypic coordinate pairs are
$(X,U)$ and $(\Xi,\Sigma)$.
Checking adjacent transpositions shows that $X\mapsto U$ and $\Xi\mapsto\Sigma$ are
equivariant isomorphisms.  The ambient norms satisfy
\[
 |\ell_E(X)|^2=2|X|^2,\quad \|\sigma_E(U)\|_F^2=|U|^2,
 \quad
 |\ell_T(\Xi)|^2=2|\Xi|^2,\quad
 \|\sigma_T(\Sigma)\|_F^2=2|\Sigma|^2.
\]
Both irreducibles are absolutely irreducible over $\R$.  Define the associated quadratic form by
\[
 \mathcal H(h):=\frac12D^2\Psi(0)[h,h].
\]
Schur's lemma therefore reduces $\mathcal H$ to one symmetric $2\times2$ multiplicity matrix
on the two copies of $E$ and one on the
two copies of $T$.  An invariant Hessian has no mixed terms between inequivalent irreducible components, so there are no $E$--$T$ cross terms.

\subsection{Scalar derivatives}
Two elementary two-parameter families determine the two diagonal and one mixed
coefficient on each isotypic component.
For the $E$-family take $(X,U)=(\varepsilon X_E,\eta U_E)$ with
$X_E=U_E=(1,-1,0)$.  For the $T$-family take
$(\Xi,\Sigma)=(\varepsilon\Xi_T,\eta\Sigma_T)$ with
$\Xi_T=\Sigma_T=(1,0,0)$.
In each case the other isotypic component is zero, so $\varepsilon$ varies the
zone weights and $\eta$ the affine metric.  Exponentiating gives
\[
\begin{aligned}
 \beta_E(\varepsilon)&=(\e^\varepsilon,\e^{-\varepsilon},1,1,
                          \e^{-\varepsilon},\e^\varepsilon), \quad &
 G_E(\eta)&=\operatorname{diag}(\e^\eta,\e^{-\eta},1),\\
 \beta_T(\varepsilon)&=(\e^\varepsilon,1,1,1,1,\e^{-\varepsilon}),\quad &
 G_T(\eta)&=\begin{pmatrix}1&0&0\\0&\cosh\eta&\sinh\eta\\
                            0&\sinh\eta&\cosh\eta\end{pmatrix}.
\end{aligned}
\]
Write $V_E:=V(\beta_E)$ and $S_E:=S(\beta_E,G_E)$, and define $V_T,S_T$
similarly.  Substitution in \eqref{eq:zonotope-formulas} gives
\[
\begin{aligned}
 V_E(\varepsilon)&=4+8\cosh\varepsilon+4\cosh(2\varepsilon),\\
 S_E(\varepsilon,\eta)
 &=2^{7/3}(1+2\cosh\varepsilon)
      \sqrt{1+2\cosh\eta}
   +2^{4/3}\bigl(1+2\cosh(2\varepsilon+\eta/2)\bigr),\\
 V_T(\varepsilon)&=8(1+\cosh\varepsilon),\\
 S_T(\varepsilon,\eta)
 &=2^{4/3}\left[
 (1+2\e^\varepsilon)\sqrt{1+2\e^\eta}
 +(1+2\e^{-\varepsilon})\sqrt{1+2\e^{-\eta}}\right]
 +2^{4/3}\bigl(1+2\sqrt{\cosh\eta}\bigr).
\end{aligned}
\]

For $\ell\in\mathcal W_0$ and $\sigma\in\mathcal G_0$, define
$\mathcal H^{\mathrm{mix}}(\ell,\sigma):=\mathcal H(\ell,\sigma)-\mathcal H(\ell,0)-\mathcal H(0,\sigma)$.
Because volume is independent of the metric and all first variations vanish at the critical
point, for $h=(\ell,\sigma)\in\mathcal W_0\oplus\mathcal G_0$,
\[
\begin{aligned}
 \mathcal H(h)&=\frac1{2V_0^{2/3}}
 \left(D_h^2S-\frac{2S_0}{3V_0}D_h^2V\right),\\
 \mathcal H^{\mathrm{mix}}(\ell,\sigma)&=\frac1{V_0^{2/3}}D_{(\ell,0)}D_{(0,\sigma)}S.
\end{aligned}
\]
Here all derivatives on the right-hand side are evaluated at the origin of the logarithmic slice.
Differentiating the restricted formulas gives the six coefficients in
Table~\ref{tab:derivatives}.  With $|X_E|^2=|U_E|^2=2$ and
$|\Xi_T|^2=|\Sigma_T|^2=1$, invariance then gives the forms below.
\begin{table}[ht]
\centering
\caption{Second derivatives and values of $\mathcal H$ in the representative directions.}
\label{tab:derivatives}
\small
\begin{tabular}{ccccc}
\toprule
module&derivative&$D^2V$&$D^2S$&\shortstack{value of $\mathcal H$ (pure) or
 $\mathcal H^{\mathrm{mix}}$ (mixed)}\\
\midrule
$E$&$\partial_{\varepsilon\varepsilon}$&$24$&$8\cdot 2^{1/3}(2+\sqrt3)$&$(5-2\sqrt3)/2^{7/3}$\\
$E$&$\partial_{\eta\eta}$&$0$&$2^{1/3}(1+4\sqrt3)$&$(1+4\sqrt3)/(8\cdot 2^{1/3})$\\
$E$&$\partial_{\varepsilon\eta}$&$0$&$4\cdot 2^{1/3}$&$2^{-1/3}$\\
\midrule
$T$&$\partial_{\varepsilon\varepsilon}$&$8$&$8\cdot 2^{1/3}\sqrt3$&$(2\sqrt3-1)/2^{7/3}$\\
$T$&$\partial_{\eta\eta}$&$0$&$2^{4/3}(3+4\sqrt3)/3$&$(3+4\sqrt3)/(12\cdot 2^{1/3})$\\
$T$&$\partial_{\varepsilon\eta}$&$0$&$8\cdot 2^{1/3}\sqrt3/3$&$2\sqrt3/(3\cdot 2^{1/3})$\\
\bottomrule
\end{tabular}
\end{table}
\begin{align*}
 \mathcal H_E(X,U)
 &=\frac{5-2\sqrt3}{2^{10/3}}|X|^2
  +\frac{1+4\sqrt3}{16\cdot 2^{1/3}}|U|^2
  +2^{-4/3}\langle X,U\rangle,\\
\mathcal H_T(\Xi,\Sigma)
 &=\frac{2\sqrt3-1}{2^{7/3}}|\Xi|^2
  +\frac{3+4\sqrt3}{12\cdot 2^{1/3}}|\Sigma|^2
  +\frac{2\sqrt3}{3\cdot 2^{1/3}}\langle\Xi,\Sigma\rangle.
\end{align*}
For the mixed terms, the off-diagonal matrix entry is half the coefficient of the displayed
inner product.  Thus, relative to the ordered pairs $(X,U)$ and $(\Xi,\Sigma)$, the two
multiplicity matrices are
\[
\begin{aligned}
 H_E&=\begin{pmatrix}
 (5-2\sqrt3)/2^{10/3}&2^{-7/3}\\
 2^{-7/3}&(1+4\sqrt3)/(16\cdot 2^{1/3})
 \end{pmatrix},\\[4pt]
 H_T&=\begin{pmatrix}
 (2\sqrt3-1)/2^{7/3}&\sqrt3/(3\cdot 2^{1/3})\\
 \sqrt3/(3\cdot 2^{1/3})&(3+4\sqrt3)/(12\cdot 2^{1/3})
 \end{pmatrix}.
\end{aligned}
\]

\subsection{Positivity and the stability estimate}

\begin{proof}[Completion of the proof of Theorem~\ref{thm:main}]
Their first diagonal entries are positive, and
\[
 \det H_E=\frac{9(2\sqrt3-3)}{2^{23/3}}>0\quad\text{and}\quad
 \det H_T=\frac{5+2\sqrt3}{12\cdot 2^{8/3}}>0.
\]
Hence $H_E$ and $H_T$ are positive definite.  By the isotypic decomposition,
$D^2\Psi(0)$ is positive definite on all ten tangent directions.  Thus the
origin is a nondegenerate strict minimum for $\Psi$ on the ten-dimensional labeled
similarity slice.
The established $C^2$ regularity and the identity $D\Psi(0)=0$ allow us to apply Taylor's theorem on the finite-dimensional
space $\mathcal W_0\oplus\mathcal G_0$, giving
\[
 \Psi(h)=\Psi(0)+\frac12D^2\Psi(0)[h,h]+o(|h|^2).
\]
After shrinking the chart, there are $c,\varepsilon>0$ such that
\begin{equation}\label{eq:chart-gap}
 \Psi(h)-\Psi(0)\geq c|h|^2 \quad
 \text{with } |h|<\varepsilon.
\end{equation}
The deficit therefore vanishes only at $h=0$ in the labeled chart.  Restoring scale, rotation, and finite relabeling yields precisely the similarity class of $\Arch$.
The support function is
\[
 h_{\Phi(\ell,\sigma)}(u)=\frac12\sum_{i<j}\e^{\ell_{ij}}
 \left|u\cdot \e^{-\sigma/2}\mb n_{ij}^{\times}\right|,
\]
for $u\in S^2$.
For volume normalization, let
$t(h):=(V_0/V(\e^\ell))^{1/3}$.  Since $\vol(\Arch)=V_0$, the volume-normalized representative is $$\widehat{\Phi(h)}=t(h)\Phi(h).$$
The factor $t$ is analytic and $t(0)=1$.  Fix $r>0$ small enough that
the closed ball $\{|h|\leq r\}$ is contained in the coordinate chart, and decrease
$\varepsilon$ so that $\varepsilon\leq r$.
Since the scalar exponential and the matrix exponential
are $C^1$,
there is $L_r>0$ such that, whenever $|h|\leq r$,
\[
 |\e^{\ell_{ij}}-1|\leq L_r|\ell_{ij}|,
 \quad
 \|\e^{-\sigma/2}-I\|_F\leq L_r\|\sigma\|_F,
 \quad
 |t(h)-1|\leq L_r|h|.
\]
Moreover, $\e^{\ell_{ij}}$ and $\|\e^{-\sigma/2}\|_F$ are uniformly
bounded on this ball.  Since $\Phi(0)=\Arch$, the inequality
$\bigl||a|-|b|\bigr|\leq|a-b|$ gives, uniformly for $u\in S^2$,
\[
\begin{aligned}
 |h_{\Phi(h)}(u)-h_{\Arch}(u)|
 &\leq \frac12\sum_{i<j}
 \Bigl(
  |\e^{\ell_{ij}}-1|
  \left|u\cdot\e^{-\sigma/2}\mb n_{ij}^{\times}\right|
  +
  \left|u\cdot
  (\e^{-\sigma/2}-I)\mb n_{ij}^{\times}\right|
 \Bigr)                                                     \\
 &\leq C_1\left(\sum_{i<j}|\ell_{ij}|+\|\sigma\|_F\right)
 \leq C_2|h|
\end{aligned}
\]
for constants $C_1,C_2>0$.  The same bounds show that
$\|h_{\Phi(h)}\|_\infty$ is uniformly bounded for $|h|\leq r$.
Consequently,
\[
\begin{aligned}
 \|h_{\widehat{\Phi(h)}}-h_{\Arch}\|_\infty
 &=
 \|t(h)h_{\Phi(h)}-h_{\Arch}\|_\infty                         \\
 &\leq
 |t(h)-1|\|h_{\Phi(h)}\|_\infty
 +\|h_{\Phi(h)}-h_{\Arch}\|_\infty              \leq C_0|h|
\end{aligned}
\]
after increasing the constant if necessary.  Hausdorff distance is the
uniform distance between support functions.  Hence, choosing the identity
rotation in the defining infimum,
\begin{equation}\label{eq:support-bound}
 d_{\mathrm{sim}}(\Phi(h),\Arch)
 \leq d_H(\widehat{\Phi(h)},\Arch)
 =\|h_{\widehat{\Phi(h)}}-h_{\Arch}\|_\infty
 \leq C_0|h|.
\end{equation}  
The Hausdorff stability estimate uses this forward bound, while inverse continuity places nearby bodies in a small coordinate neighborhood.
Take $c_{\mathrm{sim}}:=c/(2C_0^2)$. 
This choice makes the asserted inequality strict away from the Archimedean similarity class.  
By Lemma~\ref{lem:open-class} and Proposition~\ref{prop:topology}, after shrinking the Hausdorff neighborhood
$\mathcal U$, every $P\in\mathcal U$ is truncated-octahedral and its similarity class has
a lift $h$ with $|h|<\varepsilon$.  Any two such lifts differ by the orthogonal $S_4$
action, so $|h|$ is independent of the lift.  Similarity
invariance, \eqref{eq:chart-gap}, and \eqref{eq:support-bound} give
\[
 \Ical(P)-\Ical(\Arch)=\Psi(h)-\Psi(0)
 \geq c|h|^2
 \geq \frac{c}{C_0^2}d_{\mathrm{sim}}(P,\Arch)^2
 =2c_{\mathrm{sim}}d_{\mathrm{sim}}(P,\Arch)^2.
\]
This proves \eqref{eq:intrinsic-stability}.  If $P$ is not similar to $\Arch$, then
$d_{\mathrm{sim}}(P,\Arch)>0$, and the last bound is strictly larger than the right-hand
side of the asserted inequality.  If $P$ is similar to $\Arch$, both sides vanish.
\end{proof}

\section*{Acknowledgments}

The author would like to thank Thomas Hales for supervising the honors thesis this article grew out of, and Dima Arinkin for teaching the representation theory essential to the proof.
The author is grateful to anonymous referees for their valuable comments and suggestions on the manuscript.
The author also thanks the Frederick Honors College and the Painter family in Pittsburgh for their generous financial support.

\section*{Subsequent developments}

Following the resolution of the truncated octahedral conjecture by Hales--Song \cite{HalesSong2026} and independently by Cesaroni--Novaga \cite{CN26b}, this note is submitted as a historical record and for its independent methodological interest.
Its use of polar exponential coordinates, mixed second variations, and support-function estimates has been adapted to the $24$-cell in four dimensions by Novaga--Song \cite{NovagaSong2026}.
The symmetry reduction via Schur's lemma may likewise be useful for isoperimetric stability on higher-dimensional polytopes.

\end{document}